\documentclass[12pt]{article}
\usepackage[utf8]{inputenc}
\usepackage[T1]{fontenc}
\usepackage[english]{babel}
\usepackage{xcolor}
\usepackage[top=2cm, bottom=2cm, left=2cm, right=2cm]{geometry}
\usepackage{layout}
\usepackage{setspace}
\usepackage{charter}
\usepackage{url}
\usepackage{graphicx}
\usepackage{tikz}
\usetikzlibrary{matrix,arrows}
\usepackage{amsmath}
\usepackage{amssymb}
\usepackage{mathrsfs}
\usepackage{amsthm}
\usepackage{hyperref}
\usepackage{float}
\usepackage[width = 1 \textwidth]{caption}
\usepackage[format=plain,
            textfont=it]{caption}
\usepackage{titlesec}
\titleformat{\paragraph}
{\normalfont\normalsize\bfseries}{\theparagraph}{1em}{}
\titlespacing*{\paragraph}
{0pt}{3.25ex plus 1ex minus .2ex}{1.5ex plus .2ex}
\usepackage{enumitem}
\setlist{itemsep = 3 pt}
\usepackage{mathtools}
\usepackage{stmaryrd}

\newcommand{\R}{\mathbb{R}}
\newcommand{\C}{\mathbb{C}}
\renewcommand{\P}{\mathbb{P}}
\newcommand{\Z}{\mathbb{Z}}
\newcommand{\N}{\mathbb{N}}

\makeatletter
\newcommand{\specialcell}[1]{\ifmeasuring@#1\else\omit$\displaystyle#1$\ignorespaces\fi}
\makeatother

\theoremstyle{definition}
\newtheorem{lem}{Lemma}[section]
\newtheorem{prop}{Proposition}[section]
\newtheorem{defin}{Definition}[section]
\newtheorem{nota}{Notation}[section]
\newtheorem{thm}{Theorem}[section]
\newtheorem*{thm*}{Theorem}
\newtheorem{cor}{Corollary}[section]
\newtheorem{rmk}{Remark}[section]

\newtheorem{example}{Example}[section]

\let\oldbibliography\thebibliography
\renewcommand{\thebibliography}[1]{\oldbibliography{#1}
\setlength{\itemsep}{-0pt}} %Reducing spacing in the bibliography.

\title{Maximality and configuration spaces on real algebraic curves}
\author{ \small Aloïs Demory\\ \small Eberhard Karls Universität Tübingens \\\small demory@math.uni-tuebingen.de}
\date{\small September 2026}

\begin{document}

\maketitle
%\newpage

\begin{abstract}
    Let $C$ be a non-singular real algebraic curve of genus $g$. The real structure on $C$ induces a real structure on the configuration space of $n$ distinct unlabeled points on $C$. We prove that this configuration space is maximal (in the sense of the Smith-Thom inequality) if and only if $C$ is maximal or the genus of $C$ is zero and $n$ is even.
\end{abstract}

\tableofcontents

\section{Introduction}

\begin{defin}
    A non-singular real algebraic curve $C$ of genus $g$ is a non-singular complex algebraic curve $\mathbb C C$ of genus $g$ endowed with an anti-holomorphic involution $conj$. The fixed point set of $conj$ is denoted by $\R C$ and called \textit{the real part of} $C$.
\end{defin}

Each connected component of $\R C$ is homeomorphic to a circle. It has been known since the second half of the $19$-th century (see \cite{harnack}, \cite{K}) that the real part of a non-singular real algebraic curve of genus $g$ cannot have more than $g+1$ connected components. The modern generalization of this statement is the following result.

\begin{thm*}
    \textbf{(Smith-Thom inequality, see e.g. \cite{bredon})} Let $M$ be a smooth finite-dimensional manifold without boundary and endowed with a smooth involution $\tau$ with fixed point set $Fix(\tau)$. Then 
    \begin{equation*}
        \sum_{i \geq 0} \dim_{\Z_2} H^i(Fix(\tau); \Z_2) \leq \sum_{i \geq 0} H^i(M; \Z_2),
    \end{equation*}
    provided both sums are finite.
\end{thm*}

\begin{nota}
    The sum of the $\Z_2$-dimensions of the cohomology groups with coefficients in $\Z_2$ of a manifold $X$ will be denoted by $b^*(X)$ and will be abusively called the \textit{total Betti number} of $X$.
\end{nota}

\begin{defin}
    If $(M, \tau)$ satisfies $b^*(Fix(\tau)) = b^*(M)$, then $(M, \tau)$ is said to be maximal.
\end{defin}

The study of maximal real algebraic curves (in broader generality, of maximal real algebraic varieties) is a central subject in the field of topology of real algebraic varieties (see e.g. \cite{wilson}, \cite{kharlamov}). 
The investigation of the maximality of spaces of points on real algebraic curves was initiated by I. Biswas and S. D'Mello in \cite{symmetric_product_curves}. Their results were later generalized by M. Franz in \cite{symmetric_products_generalized}. Recall that the $n$-th symmetric product $Sym^n(X)$ of a manifold $X$ is the quotient of $X^n$ by the action of the symmetric group $\mathcal S_n$. If $X$ is a real algebraic variety, the space $Sym^n(\C X)$ naturally inherits a real structure from the one on $(\C X)^n$.

\begin{thm*}
    \textbf{(\cite{symmetric_products_generalized})} If $X$ is a maximal quasi-projective real algebraic variety, then $Sym^n(\C X)$ endowed with its natural real structure is maximal.
\end{thm*}

The space $Sym^n(\C X)$ can be seen as the space of configurations of $n$ unlabeled, non-necessarily distinct points on $\C X$. The goal of this note is to begin the study of maximality questions for the subspace of $Sym^n(\C X)$ formed by the configurations of $n$ unlabeled \textit{distinct} points.

\begin{nota}
    Let $M$ be a manifold and let $n$ be a positive integer. We denote by $B(M,n)$ the space of configurations of $n$ distinct unlabeled points in $M$.    
\end{nota}

Given a non-singular real algebraic curve $C$, the space $B(\C C,n)$ is naturally endowed with a real structure induced by the one on $Sym^n(\C C)$ and whose fixed point set $\R B(\C C,n)$ is the set of configurations of $n$ points in $C$ that are invariant under the anti-holomorphic involution of $C$. The goal of the text is to prove the following result.

\begin{thm}
    \label{main_thm}
    Let $g$ be a non-negative integer and let $n$ be a positive integer. Let $C$ be a non-singular real algebraic curve of genus $g$. Then the configuration space $B(\C C,n)$ endowed with its natural real structure is maximal if and only if $C$ is maximal or $g=0$ and $n$ is even.
\end{thm}

The proof consists in computing the total Betti number of $B(\C C,n)$ and $\R B(\C C,n)$, essentially using the results of \cite{fuks_configuration} and \cite{napolitano_configuration}.\\
\\
\textbf{Acknowledgements.} 
I would like to thank T. Blomme and E. Dawson for fruitful discussions and comments on an earlier version of the text. 
A. D. is supported by the Walter Benjamin Programme of the Deutsche Forschungsgemeinschaft.

\section{Preliminaries}

\subsection{Partitions in powers of $2$}

Let $n$ be a positive integer. We denote by $\mathcal{A}(n)$ the number of ways $n$ can be written as a sum of powers of $2$. It will become clear in Section \ref{subsection_cohomology_configuration} that these numbers are ubiquitous in the computations of the $\Z_2$-cohomology groups of configurations spaces on surfaces. We recall some of their well-known properties. By convention, one has $\mathcal A(0) = 1$.

\begin{lem}
    \label{lem_partitions_puissances_2}
    Let $n$ be a non-negative integer. One has
    \begin{itemize}
    \item[(1)] $\mathcal A(2n+1) = \mathcal A(2n)$;
    \item[(2)] $\mathcal A(2n) = \mathcal A(2n-1) + \mathcal A(n)$;
    \item[(3)] $\mathcal A(n) = \sum_{i=0}^{\lfloor \frac{n}{2} \rfloor} \mathcal A (i)$. \hfill \qedsymbol
\end{itemize}
\end{lem}

\subsection{Cohomology of some configuration spaces}

\label{subsection_cohomology_configuration}

We recall some statements about the topology of configuration spaces and state some of their immediate consequences.

\begin{thm}
    \label{thm_configuration_circle}
    \textbf{(see e.g. \cite{vassiliev_real_polynomials})}
    The space $B(\mathbb S^1, n)$ is homotopy equivalent to $\mathbb S^1$.
\end{thm}

\begin{thm}
    \textbf{(\cite{fuks_configuration})}
    For any non-negative integer $k$, the $\Z_2$-dimension of $H^k(B(\R^2,n); \Z_2)$ is equal to the number of ways the integer $n$ can be written as a sum of $n-k$ powers of $2$.
\end{thm}

\begin{rmk}
    As a consequence, $b^*(B(\R^2,n)) = \mathcal A(n)$.
\end{rmk}

% \begin{rmk}
%     $H^k(B(\mathbb S^2,n); \Z_2) \cong H^k(B(\R^2,n); \Z_2) \oplus H^{k-2}(B(\R^2,n-1); \Z_2)$
% \end{rmk}
\begin{thm}
    \label{thm_napolitano_compact}
    \textbf{(\cite{napolitano_configuration})} Let $S$ be a compact connected surface without boundary and let $n$ be a positive integer. Denote by $b_1$ the $\Z_2$-dimension of $H_1(S; \Z_2)$. For any non-negative integer $j$, one has
    \begin{equation*}
        H^j(B(S,n); \Z_2) \cong \bigoplus_{\substack{\ell + \sum_{i=1}^{b_1} k_i  = n\\ j' + \sum_{i=1}^{b_1} k_i = j}} H^{j'}(B(\R^2,\ell); \Z_2) \oplus \bigoplus_{\substack{1 + \ell + \sum_{i=1}^{b_1} k_i = n\\ 2 + j' + \sum_{i=1}^{b_1} k_i = j}} H^{j'} (B(\R^2, \ell); \Z_2),
    \end{equation*}
    where the numbers $k_i$, $\ell$, $j'$ are non-negative integers.
\end{thm}

\begin{example}
    We have $b^*(B(\mathbb S^2,n)) = b^*(B(\R^2,n)) + b^*(B(\R^2,n-1)) = \mathcal A(n) + \mathcal{A}(n-1)$.
\end{example}

\begin{cor}
    \label{cor_oriented_closed_surface}
    Let $g$ and $n$ be positive integers. Let $S$ be a compact orientable surface without boundary of genus $g$. One has
    % \begin{equation*}
    %     b^*(B(S,n)) = \mathcal A(n) + \sum_{k=1}^n \left[ {k + 2g-1\choose 2g-1} + {k+2g-2 \choose 2g-1} \right] \mathcal A(n-k).
    % \end{equation*}
    \begin{equation*}
    b^*(B(S,n)) = \sum_{k=0}^{\frac{n}{2}} \left[ 2{2k+2g-2 +\bar n \choose 2g-1} + {2k+2g-3 + \bar n \choose 2g-1} + {2k+2g-1 + \bar n \choose 2g-1} \right] \mathcal A(n-2k - \bar n),
\end{equation*}
where $\bar n$ is the reduction of $n$ modulo $2$.
\end{cor}

\begin{proof}
    Recall that $dim_{\Z_2} H_1(S; \Z_2) = 2g$. Using Theorem \ref{thm_napolitano_compact}, we write
    \begin{align*}
    b^*(B(S,n)) & = \sum_{k=0}^n \sum_{\sum_{i=1}^{2g}h_i = k} b^*(B(\R^2, n-k)) + \sum_{k=1}^{n-1} \sum_{\sum_{i=1}^{2g}h_i = k} b^*(B(\R^2, n-1-k))\\
    & = \sum_{k=0}^n {k+2g-1 \choose 2g-1} b^*(B(\R^2, n-k)) + \sum_{k=1}^{n-1} {k+2g-1 \choose 2g-1} b^*(B(\R^2, n-1-k))\\
    & =  b^*(B(\R^2, n)) + \sum_{k=1}^n \left[ {k+2g-1 \choose 2g-1} + {k+2g-2 \choose 2g-1} \right] b^*(B(\R^2, n-k))\\
    &= \mathcal A(n) + \sum_{k=1}^n \left[ {k + 2g-1\choose 2g-1} + {k+2g-2 \choose 2g-1} \right] \mathcal A(n-k).
\end{align*}
    Now, if $n$ is even, using property $(1)$ from Lemma \ref{lem_partitions_puissances_2}, we can write 
\begin{align*}
    b^*(B(S,n)) & = \mathcal A(n) + \sum_{k=1}^{\frac{n}{2}} \left[ {2k+2g-2 \choose 2g-1} + {2k+2g-3 \choose 2g-1} + {2k+2g-1 \choose 2g-1} + {2k+2g-2 \choose 2g-1} \right] \mathcal A(n-2k)\\
    & = \sum_{k=0}^{\frac{n}{2}} \left[ 2{2k+2g-2 \choose 2g-1} + {2k+2g-3 \choose 2g-1} + {2k+2g-1 \choose 2g-1} \right] \mathcal A(n-2k).
\end{align*}

Similarly, if $n$ is odd, we have

\begin{align*}
    b^*(B(S,n)) &= \sum_{k=0}^n \left[ {k+2g-1 \choose 2g-1} + {k+2g-2 \choose 2g-1} \right] \mathcal{A}(n-k)\\
    &= \sum_{k=0}^{\lfloor \frac{n}{2} \rfloor} \left[ {2k+2g-1 \choose 2g-1} + {2k+2g-2 \choose 2g-1} + {2k+2g \choose 2g-1} + {2k+2g-1 \choose 2g-1} \right] \mathcal A(n-2k-1)\\
    &= \sum_{k=0}^{\lfloor \frac{n}{2} \rfloor} \left[ 2{2k+2g-1 \choose 2g-1} + {2k+2g-2 \choose 2g-1} + {2k+2g \choose 2g-1} \right] \mathcal A(n-2k-1).
\end{align*}
\end{proof}

\begin{nota}
    In the rest of the text, given a non-negative integer $T$, a positive integer $g$ and $\bar n \in \{0,1 \}$, we adopt the notation
    \begin{equation*}
        C^{\bar n}(g,T) := 2{2k+2g-2 +\bar n \choose 2g-1} + {2k+2g-3 + \bar n \choose 2g-1} + {2k+2g-1 + \bar n \choose 2g-1}.
    \end{equation*}

\end{nota}

\begin{thm}
    \label{napolitano_non_compact}
    \textbf{(\cite{napolitano_configuration})} Let $S$ be a non-compact connected surface without boundary and let $n$ be a positive integer. Denote by $b_1$ the $\Z_2$-dimension of $H_1(S; \Z_2)$. For any non-negative integer $j$, one has
    \begin{equation*}
        H^j(B(S,n); \Z_2) \cong \bigoplus_{\substack{\ell + \sum_{i=1}^{b_1} k_i  = n\\ j' + \sum_{i=1}^{b_1} k_i = j}} H^{j'}(B(\R^2,\ell); \Z_2),
    \end{equation*}
    where the numbers $k_i$, $\ell$, $j'$ are non-negative integers.
\end{thm}

\begin{cor}
    \label{cor_napolitano_plane_with_holes}
    Let $p$ and $n$ be positive integers and let $S$ be the plane with $p$ distinct points removed. One has 
    \begin{equation*}
        b^*(B(S,n)) = \sum_{k=0}^n {k+p-1 \choose p-1} \mathcal A(n-k).
    \end{equation*}
\end{cor}

\begin{proof}
    Notice that $dim_{\Z_2} H_1(S;\Z_2) = p$.
    Using Theorem \ref{napolitano_non_compact}, we write
    \begin{align*}
    b^*(B(S,n)) &= \sum_{k=0}^{n} \sum_{\sum_{i=1}^p h_i = k } b^*(B(\R^2, n-k))\\
    %& = \sum_{k=0}^{n} {k+p-1 \choose p-1} b^*(B(\R^2, n-k))\\
    & = \sum_{k=0}^n {k+p-1 \choose p-1} \mathcal A(n-k).
\end{align*}
\end{proof}

\section{Configuration spaces on a real curve}

\subsection{Total Betti number of a connected component of $\R B(\C C,n)$}

\label{subsection_correspondence_vector}

Let $g$ and $n$ be positive integers. Let $C$ be a smooth real curve of genus $g$ with associated anti-holomorphic involution $conj$. Denote by $r$ the number of connected components of $\R C$ and by $C_1$, ..., $C_r$ these connected components. If $\R C = \emptyset$ and $n$ is odd, then $\R B(\C C, n)$ is empty. Otherwise, each connected component of of $\R B(\C C, n)$ is characterized by
\begin{itemize}
    \item the number $k$ of real points among the $n$ points of every configuration in this connected component;
    \item the repartition of these $k$ points on the connected components $C_1$, ..., $C_{r}$ of $\R C$, \textit{i.e.} the choice of an ordered list of non-negative integers $(m_1,...,m_{r})$ such that $\sum_{i=1}^{r} m_i = k$.
\end{itemize}

Remark that the numbers $n$ and $k$ are equal modulo $2$.

\begin{lem}
    \label{lem_cc_non_empty}
    If $\R C \neq \emptyset$, the total Betti number of the connected component of $\R B(\C C, n)$ corresponding to the vector $(m_1,...,m_r)$ is equal to 
    \begin{align*}
    \begin{cases}
        %2^{\ell} b^*(B((\C C \setminus \R C)/conj, \frac{n}{2}-k)) = 
        2^{\ell} \sum_{i=0}^{\frac{n-k}{2}} {i+g-1 \choose g-1} A(\frac{n-k}{2} - i) &\text{if $g \geq 1$};\\
        2 b^*(B(\mathbb R^2, \frac{n-k}{2})) = 2 \mathcal A(\frac{n-k}{2}) & \text{if $g=0$};
    \end{cases}
    \end{align*}
    where $\ell$ is the number of non-zero entries of the vector $(m_1,...,m_{r})$.
\end{lem}

\begin{proof}
    The connected component of $\R B(\C C, n)$ corresponding to the vector $(m_1,..., m_r)$ is homeomorphic to 
    \begin{equation*}
        \prod_{\substack{i \in \{1,...,r \}\\m_i \neq 0}} B(\mathbb S_1, m_i) \times B\left( (\C C \setminus \R C)/conj, \frac{n-k}{2}\right),
    \end{equation*}
    which is homotopy-equivalent to $\mathbb (S^1)^{\ell} \times B((\C C \setminus \R C)/conj, \frac{n-k}{2})$ by Theorem \ref{thm_configuration_circle}, where $\ell$ is the number of non-zero entries of the vector $(m_1,...,m_{r})$.
    
    As $\R C \neq \emptyset$, we have that $\C C/conj$ is a compact surface with boundary $\R C$ and with Euler characteristic $1-g$. Hence, the $\Z_2$-dimension of $H_1((\C C \setminus \R C)/conj ; \Z_2)$ is equal to $g$.

    The result now follows from Theorem \ref{napolitano_non_compact} and Künneth's theorem.
\end{proof}

\begin{lem}
    If $\R C = \emptyset$ and $n$ is even, then the total Betti number of the unique connected component of $\R B(\C C, n)$ is 

    \begin{equation*}
    b^*(B(\C C / conj, \frac{n}{2})) = \sum_{i=0}^{\frac{n}{2}} {i+g \choose g} \mathcal A(\frac{n}{2} - i) + \sum_{i=0}^{\frac{n}{2}-1} {i+g \choose g} \mathcal A(\frac{n}{2}-1-i).
\end{equation*}
\end{lem}

\begin{proof}
    As $\R C = \emptyset$, the quotient $\C C / conj$ is a compact surface without boundary and with Euler characteristic $1-g$. Hence, $\dim_{\Z_2} H_1(\C C / conj; \Z_2) = g + 1$. The result now follows from Theorem \ref{thm_napolitano_compact}.
\end{proof}

\subsection{Genus $0$}

\begin{lem}
    \label{lem_genus_zero_1}
    Let $C$ be the complex projective line endowed with its usual real structure. Let $n$ be a positive integer. The space $B(\C C, n)$ is maximal.    
\end{lem}

\begin{proof}
The total Betti number of $B(\C C, n) = B(\mathbb S^2,n)$ is
\begin{align*}
    \mathcal A(n) + \mathcal A(n-1) =
    \begin{cases}
        \mathcal A(n) + \mathcal A(n-2) &\text{ if $n$ is even};\\
        2 \mathcal A(n-1) & \text{ if $n$ is odd}.
    \end{cases}
\end{align*}

If $n$ is even, then by Lemma \ref{lem_cc_non_empty} the total Betti number of $\R B(\C C, n)$ is equal to
\begin{align*}
    \mathcal{A}(\frac{n}{2}) + \sum_{k = 1}^{\frac{n}{2}} 2 \mathcal{A}(\frac{n}{2}-k) &= \mathcal A(\frac{n}{2}) + 2 \sum_{k=0}^{\frac{n}{2}-1} \mathcal A(k)\\
    &= \sum_{k=0}^{\frac{n}{2}} \mathcal A(k) + \sum_{k=0}^{\frac{n}{2}-1} \mathcal A(k)\\
    &= \mathcal A(n) + \mathcal A(n-2).
\end{align*}

If $n$ is odd, then the total Betti number of $\R B(\C C, n)$ is equal to 

\begin{align*}
    \sum_{k=0}^{\lfloor \frac{n}{2} \rfloor} 2 \mathcal A(\lfloor \frac{n}{2} \rfloor - k) &= 2\sum_{k=0}^{\frac{n-1}{2}} \mathcal A(\frac{n-1}{2}-k)\\
    &= 2 \mathcal A(n-1).
\end{align*}

Hence, for any positive integer $n$, the configuration space $B(\C C, n)$ is maximal.
\end{proof}

\begin{lem}
    \label{lem_genus_zero_2}
    Let $C$ be a non-singular real algebraic curve of genus $0$ with empty real part and let $n$ be an even positive integer. The configuration space $B(\C C, n)$ is maximal.    
\end{lem}

\begin{proof}
 The space $\R B(\C C, n)$ is $B(\R \P^2, \frac{n}{2})$. By Theorem \ref{thm_napolitano_compact}, its total Betti number is equal to

\begin{align*}
    b^*(B(\R \P^2,\frac{n}{2})) &= \sum_{i=0}^{\frac{n}{2}} \mathcal A(\frac{n}{2}-i) + \sum_{i=0}^{\frac{n}{2}-1} \mathcal A(\frac{n}{2}-1-i)\\
    &= \mathcal A(n)+ \mathcal A(n-2)\\
    &= b^*(B(\mathbb S^2, n))
\end{align*}
\end{proof}

\begin{rmk}
    If $n$ is odd and $C$ is a smooth real algebraic curve of genus $0$ with empty real part, then $\R B(\C C, n)$ is empty, and hence $B(\C C, n)$ is not maximal.
\end{rmk}

\subsection{Higher genera: non-maximal curves}

The goal of this section is to prove the following result.
\begin{prop}
    \label{prop_non_maximal_curve}
    Let $g$ and $n$ be positive integers. Let $C$ be a non-singular non-maximal real algebraic curve of genus $g$. Then the configuration space $B(\C C, n)$ is not maximal.
\end{prop}

To prove the proposition, we distinguish two cases.

\begin{lem}
    Let $g$ and $n$ be positive integers. Let $C$ be a non-singular non-maximal real algebraic curve of genus $g$ with non-empty real part. Then the configuration space $B(\C C, n)$ is not maximal.
\end{lem}

\begin{proof}
    Let $r$ be the number of connected components of $\R C$ and denote by $C_1$, ..., $C_r$ the connected components of $\R C$.
    Let $M$ be a non-singular maximal curve of genus $g$ and denote by $M_1$, ..., $M_{g+1}$ the connected components of $\R M$.

    We define a map $\varphi$ from the set of connected components of $\R B(\C C, n)$ to the set of connected components of  $\R B(\C M,n)$.
    Recall from Section \ref{subsection_correspondence_vector} that every connected component of $\R B(\C C, n)$ (respectively, of $\R B(\C M,n)$) is uniquely characterized by a vector with $r$ (respectively, $g+1$) entries which are non-negative integers.
    The map $\varphi$ is defined as follows: the image under $\varphi$ of a connected component of $\R B(\C C, n)$ characterized by a vector $(m_1,...,m_r)$ is the connected component of $\R B(\C M,n)$ corresponding to the vector  $(m_1,...,m_r,0,...,0)$. The map $\varphi$ is clearly injective but not surjective, as $r < g+1$.

    It remains to remark that the connected component of $\R B(\C C, n)$ corresponding to the vector $(m_1,...,m_r)$ has the same total Betti number as the connected component of $\R B(\C M,n)$ associated to the vector $(m_1,...,m_r,0,...,0)$. This is a direct consequence of Lemma \ref{lem_cc_non_empty}.

    We obtain $b^*(\R B(\C C, n)) < b^*(\R B(\C M,n))$, which finishes the proof.
\end{proof}

Now we turn to the case of non-maximal real curves with empty real part.

\begin{lem}
    Let $g$ and $n$ be positive integers. Let $C$ be a non-singular non-maximal real algebraic curve of genus $g$ with empty real part. Then the configuration space $B(\C C, n)$ is not maximal.
\end{lem}

\begin{proof}
    If $n$ is odd, then $\R B(\C C, n)$ is empty, and $B(\C C, n)$ is clearly not maximal.

    Suppose that $n$ is even.
    The total Betti number of $B(\C C, n)$ and $\R B(\C C, n) = B(\C C/conj, \frac{n}{2})$ can be computed using Theorem \ref{thm_napolitano_compact}.

    On the one hand, one has
    \begin{align*}
        b^*(B(\C C, n)) = \sum_{i=0}^{n} {i+2g-1 \choose 2g-1} \mathcal A(n-i) + \sum_{i=0}^{n-1} {i+2g-1 \choose 2g-1} \mathcal A(n-1-k).
    \end{align*}

    On the other hand,

    \begin{align*}
        b^*(B(\C C/conj, \frac{n}{2})) = \sum_{i=0}^{\frac{n}{2}} {i+g \choose g} \mathcal A(\frac{n}{2}-i) + \sum_{i=0}^{\frac{n}{2}-1} {i+g \choose g} \mathcal A(\frac{n}{2}-1-i).
    \end{align*}

    Notice that for any positive integer $g$, for any positive even integer $n$, for  any non-negative integer $k \leq \frac{n}{2}$, one has ${k+g \choose g} < {\frac{n}{2}+k+2g-1 \choose 2g-1}$.

    Let $k$ be an integer such that $0 \leq k \leq \frac{n}{2}$. One then sees that the coefficient in front of $\mathcal{A}(\frac{n}{2}-k)$ in the expression of $b^*(B(\C C/conj, \frac{n}{2}))$ above is smaller than the coefficient in front of $\mathcal{A}(\frac{n}{2}-k)$ in the expression of $b^*(B(\C C, n))$. This finishes the proof.
\end{proof}

\subsection{Higher genera: maximal curves}

Let $g$ and $n$ be positive integers. In what follows, $C$ is a maximal non-singular real algebraic curve of genus $g$.

\begin{lem}
    \label{lem_rewriting_b_star}
    If $n$ is even, then $b^*(\R B(\C C, n))$ is equal to 
    \begin{equation*}
        \sum_{i = 0}^{\frac{n}{2}} {i+g-1 \choose g-1} \mathcal A(\frac{n}{2}-i) + \sum_{k=1}^{\frac{n}{2}} \sum_{\ell = 1}^{g+1} {g+1 \choose g+1-\ell} {2k-1 \choose \ell - 1} 2^{\ell} \sum_{i=0}^{\frac{n}{2}-k} {i+g-1 \choose g-1} \mathcal A(\frac{n}{2}-k-i).
    \end{equation*}
    If $n$ is odd, then $b^*(\R B(\C C, n))$ is equal to
    \begin{equation*}
        \sum_{k=0}^{\lfloor \frac{n}{2} \rfloor} \sum_{\ell = 1}^{g+1} {g+1\choose \ell} {2k \choose \ell -1} 2^{\ell} \sum_{i=0}^{\lfloor \frac{n}{2}\rfloor-k} {i+g-1 \choose g-1} \mathcal A(\lfloor \frac{n}{2} \rfloor -k-i).
    \end{equation*}
\end{lem}

\begin{proof}
    It suffices to use Lemma \ref{lem_cc_non_empty} and the correspondance between connected components of $\R B(\C C, n)$ and vectors with $g+1$ entries which are non-negative integers and whose sum is an integer smaller than or equal to $n$ and whose reduction modulo $2$ is equal to $n \mod 2$.

    If $n$ is even, we write 

    \begin{align*}
    &b^*(\R B(\C C, n)) = \sum_{k = 0}^{\frac{n}{2}} \sum_{\substack{(h_1,...,h_{g+1}) \in \N^{g+1}\\ h_1 + ... + h_{g+1} = 2k}} 2^{\# \{ i  \in \{1,...,g+1 \} \; | \; h_i \neq 0 \} } \sum_{i=0}^{\frac{n}{2}-k} {i+g-1 \choose g-1} \mathcal A(\frac{n}{2} - k - i)\\ &= \sum_{i = 0}^{\frac{n}{2}} {i+g-1 \choose g-1} \mathcal A(\frac{n}{2}-i) + \sum_{k=1}^{\frac{n}{2}} \sum_{\ell = 1}^{g+1} {g+1 \choose g+1-\ell} {2k-1 \choose \ell - 1} 2^{\ell} \sum_{i=0}^{\frac{n}{2}-k} {i+g-1 \choose g-1} \mathcal A(\frac{n}{2}-k-i)
\end{align*}

    If $n$ is odd, we write

\begin{align*}
    b^*(\R B(\C C, n)) &= \sum_{k=0}^{\lfloor \frac{n}{2} \rfloor} \sum_{\substack{(h_1,...,h_{g+1}) \in \N^{g+1} \\ h_1,...,h_{g+1} = 2k+1}} 2^{\# \{i \in \{1,...,g+1 \} \; | \; h_i \neq 0 \}} \sum_{i=0}^{\lfloor \frac{n}{2} \rfloor - k} {i+g-1 \choose g-1} \mathcal A(\lfloor \frac{n}{2} \rfloor - k - i)\\
    &= \sum_{k=0}^{\lfloor \frac{n}{2} \rfloor} \sum_{\ell = 1}^{g+1} {g+1\choose \ell} {2k \choose \ell -1} 2^{\ell} \sum_{i=0}^{\lfloor \frac{n}{2}\rfloor-k} {i+g-1 \choose g-1} \mathcal A(\lfloor \frac{n}{2} \rfloor -k-i)
\end{align*}
\end{proof}

This rewriting of $b^*(\R B(\C C, n))$ allows one to directly compare $b^*(\R B(\C C, n))$ to $b^*(B(\C C, n))$ when the genus $g$ of $C$ is equal to $1$. In this case, as $g-1= 0$, all the expressions of type ${i+g-1 \choose g-1}$ are equal to $1$. Furthermore, $\ell$ can only take the values $1$ and $2$.

\begin{lem}
    \label{lem_genus_one}
    Let $n$ be a positive integer and let $C$ be a maximal curve of genus $1$. Then $B(\C C, n)$ is maximal.
\end{lem}

\begin{proof}

    First, assume that $n$ is even. Using Lemma \ref{lem_rewriting_b_star}, we write 

    \begin{align*}
    b^*(\R B(\C C, n))&= \sum_{i=0}^{\frac{n}{2}} \mathcal A(\frac{n}{2}-i) + \sum_{k=1}^{\frac{n}{2}} \sum_{\ell=1}^2 {2 \choose 2-\ell} {2k-1 \choose \ell-1} 2^{\ell} \sum_{i=0}^{\frac{n}{2}-k} \mathcal{A}(\frac{n}{2}-k)\\ 
    &= \mathcal A(n) + \sum_{k=1}^{\frac{n}{2}} \sum_{\ell=1}^2 {2 \choose 2-\ell} {2k-1 \choose \ell-1} 2^{\ell} \mathcal A(n-2k) \\
    &= \mathcal A(n) + \sum_{k=1}^{\frac{n}{2}} (4 + (2k-1)\times 4) \mathcal A(n-2k)\\
    & = \mathcal A(n) + \sum_{k=1}^{\frac{n}{2}} 8k \mathcal A(n-2k)
\end{align*}

    On the other hand, using Corollary \ref{cor_oriented_closed_surface}, we write

    \begin{align*}
        b^*(B(\C C, n)) &= \sum_{k=0}^{\frac{n}{2}} \left[ 2{2k \choose 1} + {2k-1 \choose 1} + {2k+1 \choose 1} \right] \mathcal A(n-2k)\\
        &= \mathcal A(n) + \sum_{k=1}^{\frac{n}{2}} 8k \mathcal A(n-2k)
    \end{align*}

    Now, suppose that $n$ is odd. Using Lemma \ref{lem_rewriting_b_star}, we write

    \begin{align*}
        b^*(\R B(\C C, n)) &= \sum_{k=0}^{\lfloor \frac{n}{2} \rfloor} \sum_{\ell = 1}^2 {2 \choose \ell} {2k \choose \ell -1} 2^{\ell} \sum_{i=0}^{\lfloor \frac{n}{2} \rfloor-k} \mathcal A(\lfloor \frac{n}{2} \rfloor -k-i)\\
        &= \sum_{k=0}^{\lfloor \frac{n}{2} \rfloor} \sum_{\ell = 1}^2 {2 \choose \ell} {2k \choose \ell -1} 2^{\ell} \mathcal A(n-1-2k)\\
        &= \sum_{k=0}^{\lfloor \frac{n}{2} \rfloor} \left( 2 \times 1 \times 2 + 1 \times 2k \times 4\right) \mathcal A(n-1-2k)\\
        &= \sum_{k=0}^{\lfloor \frac{n}{2} \rfloor} (8k+4) \mathcal A(n-1-2k).
    \end{align*}

    On the other hand, applying Corollary \ref{cor_oriented_closed_surface} gives 

    \begin{align*}
        b^*(B(\C C, n)) &= \sum_{k=0}^{\lfloor \frac{n}{2} \rfloor} \left[ 2{2k+1 \choose 1} + {2k \choose 1} + {2k+2 \choose 1} \right] \mathcal A(n-2k-1)\\
        &= \sum_{k=0}^{\lfloor \frac{n}{2} \rfloor} (8k+4) \mathcal A(n-2k-1).
    \end{align*}

    This finishes the proof.
    
\end{proof}

To generalize this result to the case $g \geq 2$, for any positive integer $n$ with reduction modulo $2$ denoted by $\bar n$, for any integer $T$ such that $0 \leq T \leq \lfloor \frac{n}{2} \rfloor$, we denote by $R^{\bar n}(g,T)$ the coefficient in front of $\mathcal A(n-2k-\bar n)$ in the expression of $b^*(\R B(\C C, n))$ given in Lemma \ref{lem_rewriting_b_star} and prove that it is equal to $C^{\bar n}(g,T) = 2 {2k+2g-2 + \bar n \choose 2g-1} + {2k+2g-3+\bar n \choose 2g-1} + {2k+2g-1+ \bar n \choose 2g-1}$, which is the coefficient in front of $\mathcal A(n-2k- \bar n)$ in the expression of $b^*(B(\C C, n))$ given by Corollary \ref{cor_oriented_closed_surface}.

\begin{lem}
    Let $n$ be a positive integer, let $g$ be an integer greater than $1$ and let $T$ be an integer such that $0 \leq T \leq \lfloor \frac{n}{2} \rfloor$.

    If $n$ is even, then
    \begin{align*}
        R^{0}(g,T) = {T+g-2 \choose g-2} + \sum_{k=1}^T {T-k+g-2 \choose g-2} \sum_{\ell=1}^{g+1} {g+1 \choose \ell} {2k-1 \choose \ell - 1} 2^{\ell}.
    \end{align*}

    If $n$ is odd, then

    \begin{align*}
        R^{1}(g,T) = \sum_{k=0}^T {T-k+g-2 \choose g-2} \sum_{\ell = 1}^{g+1} {g+1 \choose \ell} {2k \choose \ell -1} 2^{\ell}.
    \end{align*}
\end{lem}

\begin{proof}
    First, suppose that $n$ is even. Notice that as $g \geq 2$, for any integer $k$ such that $0 \leq k \leq \frac{n}{2}$ one has
    \begin{align*}
        \sum_{i=0}^{\frac{n}{2}-k} {i+g-1 \choose g-1} \mathcal A(\frac{n}{2}-k-i) &= \sum_{i=0}^{\frac{n}{2}-k} \left[ {i+g-1 \choose g-1} - {i-1+g-1 \choose g-1} \right] \sum_{m=0}^{\frac{n}{2}-k-i} \mathcal A(\frac{n}{2}-k-m)\\
        & = \sum_{i=0}^{\frac{n}{2}-k} {i+g-2 \choose g-2} \mathcal A(n-2k-2i).
    \end{align*}

    Using this observation, we write 

    \begin{align*}
        &\sum_{i = 0}^{\frac{n}{2}} {i+g-1 \choose g-1} \mathcal A(\frac{n}{2}-i) + \sum_{k=1}^{\frac{n}{2}} \sum_{\ell = 1}^{g+1} {g+1 \choose g+1-\ell} {2k-1 \choose \ell - 1} 2^{\ell} \sum_{i=0}^{\frac{n}{2}-k} {i+g-1 \choose g-1} \mathcal A(\frac{n}{2}-k-i)\\   
        & = \sum_{i=0}^{\frac{n}{2}} {i+g-2 \choose g-2} \mathcal A(n-2i) + \sum_{k=1}^{\frac{n}{2}-1} \sum_{\ell = 1}^{g+1} {g+1 \choose \ell} {2k-1 \choose \ell - 1} 2^{\ell} \sum_{i=0}^{\frac{n}{2}-k} {i+g-2 \choose g-2} \mathcal A(n-2k-2i).
    \end{align*}

    To extract the coefficient in front of $\mathcal A(n-2T)$, for any $k \leq T$, we have to put $i = T - k$. This gives
    \begin{align*}
        R^0(g,T) = {T+g-2 \choose g-2} + \sum_{k=1}^T {T-k+g-2 \choose g-2} \sum_{\ell=1}^{g+1} {g+1 \choose \ell} {2k-1 \choose \ell - 1} 2^{\ell}
    \end{align*}

    The case in which $n$ is odd is obtained similarly. 
    The opening observation has to be replaced with the following. For any integer $k$ such that $0 \leq k \leq \lfloor \frac{n}{2} \rfloor$, one has

    \begin{align*}
        \sum_{i=0}^{\lfloor \frac{n}{2} \rfloor -k} {i+g-1 \choose g-1} \mathcal A(\lfloor \frac{n}{2} \rfloor-k-i) &= \sum_{i=0}^{\lfloor \frac{n}{2} \rfloor -k} \left[ {i+g-1 \choose g-1} - {i+g-2\choose g-1} \right] \sum_{m=0}^{\lfloor \frac{n}{2} \rfloor -k-i} \mathcal A(\lfloor \frac{n}{2} \rfloor-k-i-m)\\
        &= \sum_{i=0}^{\lfloor \frac{n}{2} \rfloor -k} {i+g-2 \choose g-2} \mathcal A(n-1-2k-2i).
    \end{align*}
\end{proof}

\begin{prop}
    \label{prop_genus_two_and_greater}
    Let $g$ be an integer greater than $1$, let $n$ be a positive integer and let $C$ be a maximal non-singular real algebraic curve of genus $g$. Then $B(\C C, n)$ is maximal.
\end{prop}

\begin{proof}
    It suffices to show that for any integer $T$ such that $0 \leq T \leq \lfloor \frac{n}{2} \rfloor$, one has $R^{\bar n}(g,T) = C^{\bar n}(g, T)$, where $\bar n$ is the reduction of $n$ modulo $2$.
    To prove this, we use generating functions.
    We only present the proof in the case in which $n$ is even. The case in which $n$ is odd can be treated similarly with some minor adjustments.
    For any $g \geq 2$, for any non-negative integer $k$, define 
    \begin{equation*}
        \mathcal S_k(g) =
        \begin{cases}
            \sum_{\ell=1}^{g+1} {g+1 \choose \ell} {2k-1 \choose \ell - 1} 2^{\ell} & \text{if } k \neq 0,\\
            1 & \text{if } k = 0.
        \end{cases}
    \end{equation*}
    
    One has $R(g,T) = \sum_{k=0}^{T} {T-k+g-2 \choose g-2} \mathcal S_k(g)$.

    We first notice that $\sum_{T \geq 0} R(g,T) x^T = \frac{1}{(1-x)^{g-1}} \sum_{k \geq 0} \mathcal S_k(g)$.
    
    Indeed, one has $\frac{1}{(1-x)^{g-1}} = \sum_{i\geq 0} {k+g-2\choose g-2}$.

    Hence, for any $T \geq 0$, the coefficient in front of $x^T$ in $\frac{1}{(1-x)^{g-1}} \sum_{k \geq 0} \mathcal S_k(g)$ is

    \begin{equation*}
        \sum_{k=0}^T {T-k+g-2 \choose g-2} \mathcal S_k(g).
    \end{equation*}

    The goal is now to prove that $(1-x)^{g-1} \sum_{T\geq 0} C(g, T) x^T = \sum_{k \geq 0} \mathcal S_k(g) x^k$. We first rewrite $C(g,T)$ in a form that will be easier to manipulate. Below, given an integer $i$ and a power series $P(z)$ in the variable $z$, we denote by $[z^i]P(z)$ the coefficient in front of $z^i$ in $P(z)$. 

    \begin{align*}
        C(g, T) &= 2 {2T + 2g-2 \choose 2g-1} + {2T + 2g - 3 \choose 2g-1} + {2T + 2g - 1 \choose 2g-1}\\
        &= [z^{2g-1}] \left( 2(1+z)^{2T+2g-2} + (1+z)^{2T+2g-2} + (1+z)^{2T+2g-1} \right)\\
        &= [z^{2g-1}] (1+z)^{2T+2g-3} ( 2(1+z) + 1 + (1+z)^2 )\\
        &= [z^{2g-1}] (1+z)^{2g-3} (1+z)^{2T}  (z+2)^2.
    \end{align*}

    In terms of generating functions, this means that 
    \begin{align*}
        \sum_{T \geq 0} C(g,T) x^T &= \sum_{T \geq 0} [z^{2g-1}] (1+z)^{2g-3} (1+z)^{2T}  (z+2)^2 x^T\\
        &= [z^{2g-1}] (1+z)^{2g-3} (z+2)^2 \sum_{T \geq 0} (1+z)^{2T} x^T.
    \end{align*}

    Now, for any $k \geq 0$, we want to prove that the coefficient in front of $x^k$ in $(1-x)^{g-1} \sum_{T\geq 0} C(g, T) x^T$ is equal to $\mathcal S_k(g)$. When $k=0$, it is immediate: 
    
    \begin{equation*}
        [x^0] (1-x)^{g-1} \sum_{T\geq 0} C(g, T) x^T = C(g,0) = 2 {2g-2 \choose 2g-1} + {2g - 3 \choose 2g-1} + {2g - 1 \choose 2g-1} = 1.
    \end{equation*}
    Let us fix a positive integer $k$. We want to extract the coefficient in front of $x^k$ in the expression $(1-x)^{g-1} \sum_{T\geq 0} C(g, T) x^T$. First, notice that

    \begin{align*}
        (1-x)^{g-1} \sum_{T\geq 0} C(g, T) &= \left( \sum_{j=0}^{g-1} (-1)^j {g-1 \choose j} x^j \right) [z^{2g-1}] (1+z)^{2g-3} (z+2)^2 \sum_{T \geq 0} (1+z)^{2T} x^T\\
        &= [z^{2g-1}] (1+z)^{2g-3} (z+2)^2 \left( \sum_{j=0}^{g-1} (-1)^j {g-1 \choose j} x^j \right) \left( \sum_{T \geq 0} (1+z)^{2T} x^T \right).
    \end{align*}

    Hence, we can write

    \begin{align*}
        [x^k](1-x)^{g-1} \sum_{T\geq 0} C(g, T) x^T = [z^{2g-1}] (z+2)^2 (1+z)^{2g-3} \sum_{j=0}^{\min(k, g-1)} (-1)^j {g-1 \choose j} (1+z)^{2(k-j)}.
    \end{align*}

    If $j >k$, then $(z+2)^2 (1+z)^{2g-3 + 2(k-j)} (-1)^j {g-1 \choose j}$ is of degree at most $2g-3 < 2g-1$ in $z$. Hence we can replace $\min(k, g-1)$ with $g-1$ in the previous expression.

    \begin{align*}
        [x^k](1-x)^{g-1} \sum_{T\geq 0} C(g, T) x^T &= [z^{2g-1}] (z+2)^2 (1+z)^{2g-3} \sum_{j=0}^{g-1} (-1)^j {g-1 \choose j} (1+z)^{2(k-j)}\\
        &= [z^{2g-1}] (z+2)^2 (1+z)^{2g-2k-3} \sum_{j=0}^{g-1} (-1)^j {g-1 \choose j} (1+z)^{-2j}\\
        &= [z^{2g-1}] (z+2)^2 (1+z)^{2g-2k-3} (1 - (1+z)^{-2})^{g-1}.
    \end{align*}

    Using $1- (1+z)^{-2} = \frac{(1-z)^2 -1}{(1-z)^2} = \frac{z(z+2)}{(1+z)^2}$, we write

    \begin{align*}
        [x^k](1-x)^{g-1} \sum_{T\geq 0} C(g, T) x^T &= [z^{2g-1}] (z+2)^2 (1+z)^{2g-2k-3} \frac{z^{g-1}(z+2)^{g-1}}{(1+z)^{2g-2}}\\
        &= [z^{2g-1}] z^{g-1} (z+2)^{g+1} (1+z)^{2k-1}\\
        &= [z^{g}] (z+2)^{g+1} (1+z)^{2k-1}\\
        &= [z^{g}] \left( \sum_{r=0}^{g+1} {g+1 \choose r} 2^{g+1-r} z^r \right) \left( \sum_{s=0}^{2k-1} {2k-1 \choose s} z^s \right)\\
        &= \sum_{r=0}^g {g+1 \choose r} 2^{g+1-r} {2k-1 \choose g-r}.
    \end{align*}

    It remains to perform the change of variables $\ell = g+1-r$ to obtain

    \begin{align*}
        [x^k](1-x)^{g-1} \sum_{T\geq 0} C(g, T) x^T = \sum_{\ell = 1}^{g+1} {g+1 \choose \ell} 2^{\ell} {2k-1 \choose \ell-1}.
    \end{align*}
\end{proof}

    We can now put everything together to prove Theorem \ref{main_thm}.

\begin{proof}[Proof of Theorem \ref{main_thm}]
    Let $C$ be a non-singular real algebraic curve of genus $g$.

    If $g=0$, then Lemmas \ref{lem_genus_zero_1} and \ref{lem_genus_zero_2} state that $B(\C C, n)$ is maximal if and only if $C$ is maximal or $n$ is even.

    If $g>0$ and $C$ is not maximal, then Proposition \ref{prop_non_maximal_curve} ensures that $B(\C C, n)$ is not maximal.

    If $g >0$ and $C$ is maximal, then Lemma \ref{lem_genus_one} and Proposition \ref{prop_genus_two_and_greater} state that $B(\C C, n)$ is maximal.
\end{proof}

\bibliographystyle{alpha}
\bibliography{biblio}

@article
{wilson,
author =       "G. Wilson",
title =        "Hilbert's sixteenth problem",
journal =      "\textrm{Topology}",
volume =       "17",
number =       "1",
pages =        "53--73",
year =         "1978",
DOI =          "https://doi.org/10.1016/0040-9383(78)90012-5"
}

@article{kharlamov,
	doi = {10.1070/rm2000v055n04abeh000315},
	url = {https://doi.org/10.1070%2Frm2000v055n04abeh000315},
	year = 2000,
	month = {aug},
	journal = {{IOP} Publishing},
	volume = {55},
	number = {4},
	pages = {735--814},
	author = {A. I. Degtyarev and V. M. Kharlamov},
	title = {Topological properties of real algebraic varieties: du cot{\'{e}
} de chez {R}okhlin},
  
	journal = {Russian Mathematical Surveys}
}

@article{harnack, 
    title="{Über die Vieltheiligkeit der ebenen algebraischen Curven}", 
    volume={10}, 
    journal={Mathematische Annalen}, 
    author={A. Harnack}, 
    year={1876}, 
    pages={189–198}
}

@article{bredon,
  title={Introduction to compact transformation groups},
  author={Bredon, G. E.},
  journal = {Academic Press, N.Y. - London},
  year={1972}
}

@misc{K,
 author = {Klein, F.},
 title = {Gesammelte mathematische {Abhandlungen}. {Dritter} {Band}: {Elliptische} {Funktionen}, insbesondere {Modulfunktionen}. {Hyperelliptische} und {Abelsche} {Funktionen}. {Riemannsche} {Funktionentheorie} und automorphe {Funktionen}. {Anhang}. {Herausgegeben} von {R}. {Fricke}, {H}. {Vermeil} und {E}. {Bessel}-{Hagen}. {Reprint}},
 year = {1973},
 language = {German},
 howpublished = {Berlin-Heidelberg-New {York}: {Springer}-{Verlag}. {IX}, 774 {S}. mit 138 {Textfig}., {Anhang} 35 {S}. {Preis} f{\"u}r alle 3 {B{\"a}nde} {DM} 210.00; \$ 77.70 (1973).},
 zbMATH = {3423988},
 Zbl = {0269.01017}
}

@article{fuks_configuration,
    author = {Fuks, D. B.},
    title = {Cohomology of the braid group mod 2},
    journal = {Funkcional. Anal.
i Prilov zen.},
    volume = {4.2},
    pages = {62--73},
    year = {1970}
}

@article{napolitano_configuration,
    author = {Napolitano, F.},
    title = {Configuration spaces on surfaces},
    journal = {C. R. Acad. Sci. Paris},
    volume = {327},
    issue = {1},
    pages = {887--892},
    year = {1998}
}

@article{vassiliev_real_polynomials,
    author = {Vassiliev, V. A.},
    title = {Homology of spaces of homogeneous polynomials in {$\mathbb R^2$} without multiple zeros},
    journal = {ArXiv preprint},
    year = {2014}
}

@article{symmetric_product_curves,
    author = {Biswas, I. and D'Mello, S.},
    title = {M-curves and symmetric products},
    journal = {Proc. Indian Acad. Sci. (Math. Sci.)},
    volume = 127,
    issue = 4,
    pages = {615--624},
    year = {2017}
}

@article{symmetric_products_generalized, 
    title={Symmetric Products of Equivariantly Formal Spaces},
    volume={61}, 
    DOI={10.4153/CMB-2017-032-0}, 
    number={2}, 
    journal={Canadian Mathematical Bulletin}, 
    author={Franz, Matthias}, 
    year={2018}, 
    pages={272–281}
    }

\end{document}